\documentclass[12pt]{article}
\pdfoutput=1
\usepackage{amssymb, amsmath, amsthm, epsf, epsfig, color}
\usepackage{alltt}
\usepackage[usenames,dvipsnames]{xcolor}
\usepackage{fullpage}
\usepackage{multicol}
\usepackage{caption}

\newlength{\originalbase}
\newcommand{\spacing}[1]{\setlength{\baselineskip}{#1\originalbase}}

\begin{document}
\spacing{1.5}

\newtheorem{theorem}{Theorem}[section]
\newtheorem{claim}[]{Claim}
\newtheorem{prop}[theorem]{Proposition}
\newtheorem{lemma}[theorem]{Lemma}
\newtheorem{corollary}[theorem]{Corollary}
\newtheorem{conjecture}[theorem]{Conjecture}
\newtheorem{defn}[theorem]{Definition}
\newtheorem{example}[theorem]{Example}

\title{On the number of 5-cycles in a tournament}

\author{Natasha Komarov\thanks{Dept.\ of Math, CS, and Stats, St.\ Lawrence University, Canton NY 13617, USA; nkomarov@stlawu.edu.}
\, and John Mackey\thanks{Dept.\ of Math., Carnegie Mellon University, Pittsburgh PA 15213, USA; jmackey@andrew.cmu.edu.}}

\maketitle

\begin{abstract}
We find a formula for the number of directed 5-cycles in a tournament in terms of its edge scores and use the formula to find upper and lower bounds on the number of 5-cycles in any $n$-tournament. In particular, we show that the maximum number of 5-cycles is asymptotically equal to $\frac{3}{4}{n \choose 5}$, the expected number 5-cycles in a random tournament ($p=\frac{1}{2}$), with equality (up to order of magnitude) for almost all tournaments. 
\end{abstract}

\section{Introduction}
\label{intro section}

The work of Beineke and Harary~\cite{BeinekeHarary} bounds the number of strong $k$-subtournaments in any $n$-tournament for $k=3,4,5$, and consequently the number of $k$-cycles for $k=3,4$. 
 David Berman~\cite{Berman,BermanThesis}, maximized the number of $5$-cycles in a narrow family of tournaments (specifically, those that are ``semi-transitive''). More recently, Savchenko~\cite{Savchenko} has established bounds on 5-cycles and 6-cycles in regular tournaments.
 Computing the number of 5-cycles in a general $n$-tournament has remained elusive, however.
We find an exact formula for the number of 5-cycles in an $n$-tournament in terms of its edge scores
and use this result to derive upper and lower bounds on the number of 5-cycles.

It is interesting to note that for $k=4$, the maximum number of $k$-cycles is greater than the expected number in a random tournament with edge probability $p=\frac{1}{2}$ (by a factor of $\frac{4}{3}$) whereas for $k=3$ and $k=5$ these values are asymptotically equal.

\subsection{Context and motivation}


In extremal combinatorics we are frequently interested in determining whether the largest or smallest possible number of copies of a given object
in a graph or tournament is asymptotically the same as the expected number of copies of it in a random graph or tournament.

Perhaps the first result in this direction was Goodman's Theorem (initially stated and proven by Goodman~\cite{Goodman}, with the proof later improved upon by Lorden~\cite{Lorden}), which
states that the number of complete 3-vertex subgraphs plus the number of 3-vertex independent sets in an $n$-vertex graph is at least $n(n-1)(n-5)/24$,
whereas the expected number of such objects in a random graph (with edge density $p=\frac{1}{2}$) on $n$ vertices is $n(n-1)(n-2)/24$.

This led to the conjecture of Burr and Rosta~\cite{BurrRostaConjecture} (extending a conjecture of Erdos~\cite{ErdosConjecture}) that the sum of the number of complete $k$-vertex subgraphs and the number of $k$-vertex independent sets is minimized at about ${n \choose k} 2^{1- {k \choose 2}} $, which is the expected number of such occurences in a random $\left(p=\frac{1}{2}\right)$ $n$-vertex graph. Thomason~\cite{Thomason} disproved this conjecture for all $k \geq 4$, but other positive results similar to Goodman's exist (e.g.\ \cite{ApproxOfSidorenko,MultOfSGs}).

In the setting of tournaments, it was shown by Moon~\cite{MoonThm}, that the number of acyclic subtournaments on $k$ vertices in an $n$-vertex tournament is at least

$$\frac{1}{2^{k \choose 2}}\prod_{j=0}^{k-1}(n - 2^{j} + 1),$$

which is asymptotically the same as the expected number of such occurrences in a random $n$-vertex tournament.

For $k{=}3$ the result above was initially discovered by Kendall and Babington Smith~\cite{KendallSmith} using the method of paired comparisons in the context of maximizing the number of 3-cycles in an $n$-vertex tournament. We see that an $n$-vertex tournament will contain no more than $\frac{1}{24} n(n+1)(n-1)$ 3-cycles when $n$ is odd and $\frac{1}{24}n(n+2)(n-2)$ 3-cycles when $n$ is even (with equality holding if and only if the tournament is regular; see, e.g.\, \cite{BeinekeHarary, ReidBeineke, KendallSmith, MoonApp, Moon3cycles}), which is approximately the number of 3-cycles that one expects in a random $n$-vertex tournament.

For $k{=}4$, the work of Beineke and Harary~\cite{BeinekeHarary} shows that there can be no more than $\frac{1}{48} n(n+1)(n-1)(n-3)$ 4-cycles in a tournament on $n$ vertices if $n$ is odd and no more than $\frac{1}{48}n(n+2)(n-2)(n-3)$ if $n$ is even, and moreover, that this number can be achieved by a particular family of tournaments. (See also the work of K.\ B.\ Reid on this topic~\cite{KBReid-1989}.)

One might expect, just as in the case of Thomason's disproof of Erdos' conjecture, that the maximum number of $k$-cycles in an $n$-vertex tournament would be asymptotically larger than the expected number of $k$-cycles in a random $n$-vertex tournament for all $k \geq 4$.
As a result of our work, however, we see that the maximum number of 5-cycles in an $n$-vertex tournament is asymptotically the same as the expected number of 5-cycles in a random $n$-vertex tournament.

\section{The number of 5-cycles in a tournament}

The expected number of (directed) 5-cycles in an $n$-vertex tournament is given by
$ \frac{3}{4}{n \choose 5} $. Let $c(T,k)$ be the number of $k$-cycles in a tournament $T$. We will find $c(T,5)$ for any tournament $T$ in terms of its edge scores, and show that the maximum number of 5-cycles in a tournament is always (asymptotically) at most the expected number.

\subsection{The number of 5-cycles in a tournament}

The {\bf edge scores} of a tournament $T = (V,E)$ are the ordered 4-tuples $ (A(u,v),B(u,v),C(u,v),D(u,v))$ where we define 
 
\begin{itemize}
	\item $A(u,v) = |\{w \in V\backslash\{u,v\}\} | (u,w) \in E \mbox{ and } (v,w) \in E \} |$
	 (i.e.\ the number of vertices that both $u$ and $v$ have as out-neighbors)
	\item $B(u,v) = |\{w \in V\backslash\{u,v\}\} | (w, u) \in E \mbox{ and } (w, v) \in E \} |$
		
	(i.e.\ the number of vertices that both $u$ and $v$ have as in-neighbors)
	\item $C(u,v) = |\{w \in V\backslash\{u,v\}\} | (u,w) \in E \mbox{ and } (w, v) \in E \} |$
	
	(i.e.\ the number of vertices that are out-neighbors of $u$ and in-neighbors of $v$)
	\item $D(u,v) = |\{w \in V\backslash\{u,v\}\} | (w,u) \in E \mbox{ and } (v, w) \in E \} |$

	(i.e.\ the number of vertices that form a directed 3-cycle with $u$ and $v$)
\end{itemize}

		\begin{figure}
		
	\begin{center}
		\begin{tabular}{cccc}

		\begin{minipage}{.17\textwidth}
			\includegraphics[scale=.25]{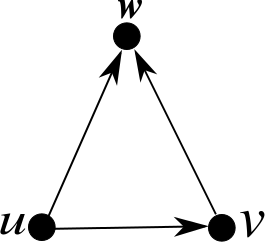}
		\end{minipage}
			
			&
			
		\begin{minipage}{.17\textwidth}
			\includegraphics[scale=.25]{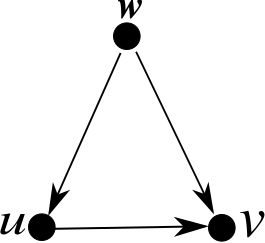}
		\end{minipage}
		
			&
			
		\begin{minipage}{.17\textwidth}
			\includegraphics[scale=.25]{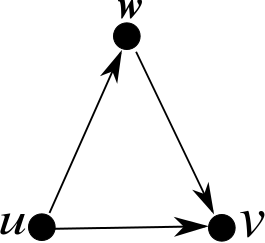}
		\end{minipage}
		
			&
			
		\begin{minipage}{.17\textwidth}
			\includegraphics[scale=.25]{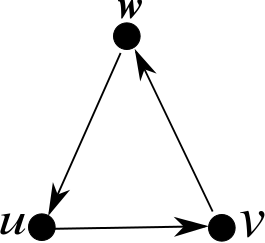}
		\end{minipage} 
		
		\end{tabular}
	\end{center}	
		
	\captionsetup{width=0.6\textwidth}

		\caption{Visual representations of the vertices counted by 
				(from~left~to~right)~$A(u,v), B(u,v), C(u,v), D(u,v)$}
		\end{figure}

	When there is no possibility of confusion, we will shorten these to simply $A, B, C,$ and $D$. 
	
Note that for any edge $(u,v) \in E$,
\begin{eqnarray}
	\label{odu}
	od(u) &=& 1 + A(u,v) + C(u,v)\\
	\label{idu}
	id(u) &=& B(u,v) + D(u,v)\\
	\label{odv}
	od(v) &=& A(u,v) + D(u,v)\\
	\label{idv}
	id(v) &=& 1 + B(u,v) + C(u,v)\\
	\label{sums to n-2}
	n{-}2&=& A(u,v) + B(u,v) + C(u,v) + D(u,v)
\end{eqnarray}

\begin{theorem}
\label{exact count}
The number of 5-cycles in an $n$-tournament $T=(V,E)$ with edge scores $(A(u,v),B(u,v),C(u,v),D(u,v))_{(u,v) \in E}$ is given by
\begin{equation*}
c(T,5) = \frac{3}{4}{n \choose 5} - \frac{1}{8}\sum_{(u,v)\in E} [(C{+}D)(A{-}B)^2 + (A{+}B)(C{-}D)^2] + \frac{1}{4}\sum_{(u,v)\in E} (A{+}B)(C{+}D).
\end{equation*}
where, for notational convenience, $A {=} A(u,v), B {=} B(u,v), C{=}C(u,v)$, and $D {=} D(u,v)$.
\end{theorem}

\begin{proof}

\begin{figure}
\includegraphics[scale=.7]{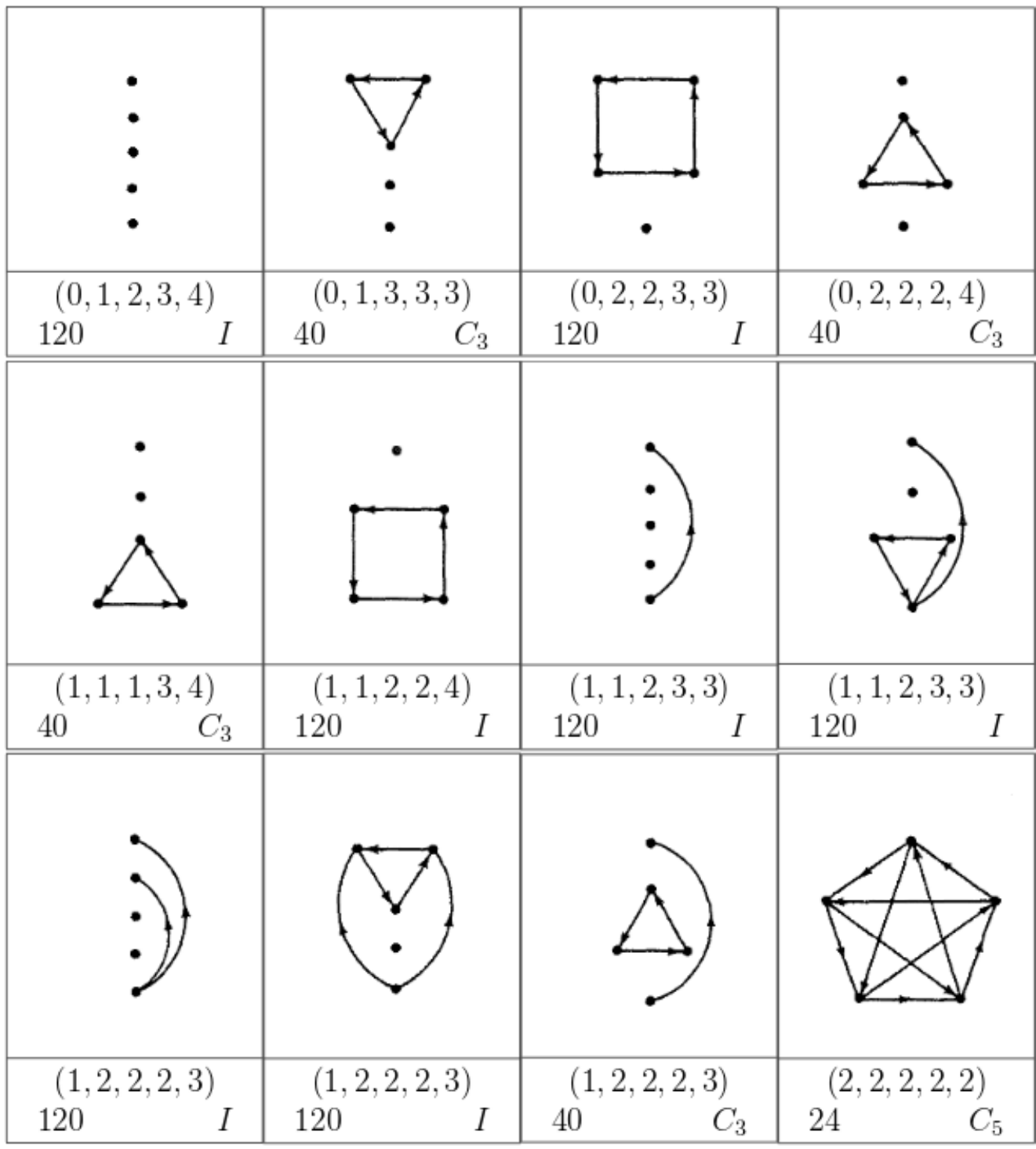}
\caption{The 12 non-isomorphic tournaments on 5 vertices; image taken from~\cite{MoonApp}.}
\label{noniso5tourns}
\end{figure}

There are twelve non-isomorphic tournaments on five vertices, displayed in Figure~\ref{noniso5tourns}. In the figure, whenever an arc is omitted between a pair of vertices, 
it goes from the higher vertex to the lower vertex, as in~\cite{MoonApp}. The number on the lower left in each box is the number of ways of
labeling that tournament's vertices and the symbol in the lower right in each box denotes that tournament's automorphism group.
We will call these tournaments $T_1$ through $T_{12}$ (in the order in which they are displayed).

Let $T=(V,E)$ be an arbitrary tournament on $n$ vertices.
Let $A_i(T)$ be the number of appearances of $T_i$ as an induced subtournament in $T$, for each $i \in [12]$.
We will write $A_i(T) = A_i$ when this will not result in any ambiguity. 
 Note that
\begin{equation}
\label{basic}
{n \choose 5} = \sum_{i=1}^{12} A_i
\end{equation}

and

\begin{equation}
\label{num5cycs}
c(T,5) = A_7 + A_8 + A_9 + 2A_{10}+3A_{11} + 2 A_{12}
\end{equation}
where the coefficients on the right hand side of Equation~\ref{num5cycs} are the number of copies of 5-cycles in the corresponding tournaments. The reader is encouraged to verify that tournaments $T_1$ through $T_6$ contain no directed 5-cycles, tournaments $T_7$ through $T_9$ each contain exactly one, $T_{10}$ and $T_{12}$ each contain exactly two, and $T_{11}$ contains exactly three.

As we did for ${n \choose 5}$ and $c(T,5)$ in Equations (\ref{basic} and \ref{num5cycs}), we will write linear relations for twelve quantities involving edge scores in terms of $A_1, A_2, \dots, A_{12}$. The 12 equations involving sums
of edge scores are verified in Section~\ref{section:appendix} using indicator functions.
We summarize these linear relations in the 14 by 12 matrix shown in Figure~\ref{matrix relations}, in which row $i$ of the matrix gives the coefficients of the $A_i$'s so that their sum yields the sum given in item $i$ in the following list:

\begin{enumerate}
\item $\displaystyle \sum_{(u,v) \in E} {A(u,v) \choose 2} C(u,v)$
\item $\displaystyle \sum_{(u,v) \in E} {A(u,v) \choose 2} D(u,v)$
\item $\displaystyle \sum_{(u,v) \in E} {B(u,v) \choose 2} C(u,v)$
\item $\displaystyle \sum_{(u,v) \in E} {B(u,v) \choose 2} D(u,v)$
\item $\displaystyle \sum_{(u,v) \in E} {C(u,v) \choose 2} A(u,v)$
\item $\displaystyle \sum_{(u,v) \in E} {C(u,v) \choose 2} B(u,v)$
\item $\displaystyle \sum_{(u,v) \in E} {D(u,v) \choose 2} A(u,v)$
\item $\displaystyle \sum_{(u,v) \in E} {D(u,v) \choose 2} B(u,v)$
\item $\displaystyle \sum_{(u,v) \in E} A(u,v) B(u,v) C(u,v)$
\item $\displaystyle \sum_{(u,v) \in E} A(u,v) B(u,v) D(u,v)$
\item $\displaystyle \sum_{(u,v) \in E} A(u,v) C(u,v) D(u,v)$
\item $\displaystyle \sum_{(u,v) \in E} B(u,v) C(u,v) D(u,v)$
\item $\displaystyle {n \choose 5}$
\item $\displaystyle c(T,5)$
\end{enumerate}

\begin{figure}[h]
\centering 

$\left[\begin{tabular}{cccccccccccc}
1 & 0 & 0 & 3 & 0& 2& 0& 0& 0& 0& 0& 0\\

0& 3& 1& 0& 0& 0& 1& 1& 0& 0& 0& 0\\

1& 0& 2& 3& 0& 0& 0& 0& 0& 0& 0& 0\\

0& 0& 0& 0& 3& 1& 1& 1& 0& 0& 0& 0\\

1& 0& 0& 0& 3& 2& 0& 0& 0& 0& 0& 0 \\

1& 3& 2& 0& 0& 0& 0& 0& 0& 0& 0& 0\\

0& 0& 1& 0& 0& 0& 0& 1& 1& 1& 0& 0\\

0& 0& 0& 0& 0& 1& 0& 1& 1& 1& 0& 0\\

1& 3& 0& 0& 3& 0& 1& 3& 0& 1& 0& 0\\

0& 0& 1& 3& 0& 1& 0& 0& 1& 2& 3& 5\\

0& 0& 2& 0& 0& 0& 1& 1& 1& 2& 3& 0\\

0& 0& 0& 0& 0& 2& 1& 1& 1& 2& 3& 0\\

1& 1& 1& 1& 1& 1& 1& 1& 1& 1& 1& 1\\

0& 0& 0& 0& 0& 0& 1& 1& 1& 2& 3& 2 
\end{tabular}\right]$

\caption{14 linear relations in $A_1$ through $A_{12}$}

\label{matrix relations}
\end{figure}

From this matrix, we arrive at the following conclusion: 
 $$8R_{14} =  6R_{13} - 2\sum_{i=1}^8 R_i + 2\sum_{i=9}^{12}R_i ,$$

where $R_i$ is the $i^{th}$ row of the matrix. This yields

$$8c(T,5) = 6{n \choose 5} 
- \sum_{(u,v) \in E} \left\{  2 \left({A \choose 2} C + {A \choose 2} D + {B \choose 2} C + {B \choose 2} D \right. \right.$$ 
$$\left. \left.+ {C \choose 2} A + {C \choose 2} B + {D \choose 2} A + {D \choose 2} B\right)
- 2(A B C + A B D + A C D + B C D)\right\}$$

which is

$$= 6{n \choose 5} 
- \sum_{(u,v) \in E} \left\{ (A^2 C + A^2 D + B^2 C + B^2 D - 2 A B C - 2 A B D) \right.$$ 
$$\left.+ (C^2 A + C^2 B + D^2 A + D^2 B - 2 A C D - 2 B C D)
- 2(A C + A D + B C + B D) \right\}$$

which is

$$= 6{n \choose 5} 
- \sum_{(u,v) \in E} \{ (C + D)(A^2 - 2 A B + B^2)$$ 
$$+ (A + B)(C^2 - 2 C D + D^2)
- 2(A C + A D + B C + B D)\}.$$

Upon factoring, this yields the following identity

\begin{equation*}
8c(T,5) = 6{n \choose 5} - \sum_{(u,v)\in E} [(C{+}D)(A{-}B)^2 + (A{+}B)(C{-}D)^2] + 2\sum_{(u,v)\in E} (A{+}B)(C{+}D),
\end{equation*}

as desired.
\end{proof}

The sum being subtracted is nonnegative and the sum being added is a lower-order term. Therefore, 

\begin{equation*}
c(T,5) \le \frac{3}{4} {n \choose 5}  + O(n^4)
\end{equation*}

with equality if and only if 
\begin{equation}
\label{equality condition}
 \sum_{(u,v)\in E} [(C{+}D)(A{-}B)^2 + (A{+}B)(C{-}D)^2] = O(n^4)
\end{equation}

Therefore, we have as corollaries to Theorem~\ref{exact count} the following bounds.

\begin{corollary}
\label{upper bound corollary}
For all $n$-tournaments $T$, 
$$c(T,5) \le \frac{3}{4} {n \choose 5} + \frac{1}{4}{n \choose 2}\left(\frac{n-2}{2} \right)^2. $$
\end{corollary}

\begin{proof}
The sum being subtracted in the statement of Theorem~\ref{exact count} is at most zero, so we focus on the quantity
$$ \sum_{(u,v) \in E} (A(u,v)+B(u,v))(C(u,v)+D(u,v)).$$

Recall that $A(u,v)+B(u,v)+C(u,v)+D(u,v) = n{-}2$ for each $(u,v) \in E$, so $(A(u,v)+B(u,v))(C(u,v)+D(u,v))$ is maximized when $A(u,v)+B(u,v) = C(u,v)+D(u,v) = \frac{n-2}{2}$. Therefore 

$$ \sum_{(u,v) \in E} (A(u,v)+B(u,v))(C(u,v)+D(u,v)) \le {n \choose 2} \left( \frac{n-2}{2}\right)^2.$$
\end{proof}

\begin{corollary}

\label{lower bound corollary}
For all $n$-tournaments $T$, 
$$ c(T,5) \geq \frac{3}{4}{n \choose 5} - \frac{1}{2}{n-2 \choose 2}\sum_{w \in V} \left(od(w) - \frac{n-1}{2} \right)^2  
- \frac{3}{8}{n \choose 3}.$$
\end{corollary}

\begin{proof}
Note that 
\begin{equation}
\label{lower bound sum}
c(T,5)\geq \frac{3}{4}{n \choose 5} - \frac{1}{8}\sum_{(u,v)\in E} [(C{+}D)(A{-}B)^2 + (A{+}B)(C{-}D)^2],
\end{equation}
so we seek an upper bound for the sum on the right hand side of (\ref{lower bound sum}).

From Equations (\ref{odu}), (\ref{idu}), (\ref{odv}), and (\ref{idv}) above, we see that 
\begin{eqnarray*}
	A{-}B &=& od(v) - id(u), \mbox{ and}\\
	C{-}D &=& od(u) - od(v) - 1
\end{eqnarray*}

We also note that $C+D \le n-2$, $A+B \le n-2$, and $id(u) = n-1-od(u)$ for any $u \in V$.
Therefore, 
\begin{equation}
\label{subtracted part}
 \sum_{(u,v)\in E} [(C{+}D)(A{-}B)^2 + (A{+}B)(C{-}D)^2]
\end{equation}
 is bounded above by 
 
\begin{eqnarray*}
&&  (n-2)\sum_{(u,v)\in E} [(od(v) {-} id(u))^2 + (od(u) {-} od(v) {-} 1)^2] \\
&=& (n-2)\sum_{(u,v)\in E}[2od(v)^2  {+} id(u)^2  {+} od(u)^2 - 2od(v)id(u) - 2od(u)od(v)  {+} 1 - 2 od(u)  {+} 2 od(v))]\\
&=& (n-2)\sum_{(u,v)\in E}[2od(v)^2 + id(u)^2 + od(u)^2 - 2(n{-}1) \, od(v) + 1 - 2 od(u) + 2 od(v))]\\
&=& (n-2)\sum_{(u,v)\in E}[2od(v)^2 + (n-1-od(u))^2 + od(u)^2 - 2(n{-}1) \, od(v) + 1 - 2 od(u) + 2 od(v))]\\
&=& (n-2)\sum_{(u,v)\in E}[2od(v)^2  {+} (n{-}1)^2 {-} 2(n{-}1)od(u)  {+} 2od(u)^2 {-} 2(n{-}1) \, od(v)  {+} 1 {-}2 od(u)  {+} 2 od(v))]	\\
&=& (n-2)\sum_{(u,v)\in E} \left[2 \left(od(v){-}\frac{n{-}1}{2}\right)^2 {+} 2 \left(od(u){-}\frac{n{-}1}{2}\right)^2 {+} 1 {-} 2 od(u) {+} 2 od(v))
 \right]	\\
\end{eqnarray*}

We can translate this sum over edges to a sum over vertices. If $f$ is any function, then summing $f(v)$ over all edges $(u,v)$ means that for each time that a vertex $v$ appears as the terminus of a directed edge (which happens $id(v)$ times), it contributes $f(v)$ to the sum. Therefore $\displaystyle \sum_{(u,v)\in E} f(v) = \sum_{v \in V} id(v)f(v)$. 

Summing $f(u)$ over all edges $(u,v)$ means that for each time that a vertex $u$ appears as the origin of a directed edge (which happens $od(u)$ times), it contributes $f(u)$ to the sum. Therefore
$\displaystyle \sum_{(u,v)\in E} f(u) = \sum_{u \in V} od(u)f(u)$. Hence

$$\sum_{(u,v)\in E} 2 \left(od(v){-}\frac{n{-}1}{2}\right)^2
=\sum_{w \in V} 2id(w) \left(od(w) - \frac{n-1}{2} \right)^2,$$

$$\sum_{(u,v)\in E} 2 \left(od(u){-}\frac{n{-}1}{2}\right)^2
= \sum_{w \in V} 2od(w) \left(od(w) - \frac{n-1}{2} \right)^2,$$

$$\sum_{(u,v)\in E} 1 
= \sum_{w \in V} \frac{n-1}{2},$$

$$\sum_{(u,v)\in E} 2 od(u) = \sum_{w \in V} 2 od(w)^2, \text{   and}$$

$$\sum_{(u,v)\in E} 2 od(v) = \sum_{w \in V} 2 id(w) od(w)$$

and the bound above is

\begin{eqnarray*}
&=& (n-2)\sum_{w \in V}[ 2id(w) \left(od(w) - \frac{n-1}{2} \right)^2 + 2od(w) \left(od(w) - \frac{n-1}{2} \right)^2 \\
&& \,\,\,\, \,\,\,\, \,\,\,\, \,\,\,\, \,\,\,\, \,\,\,\,  +\frac{n-1}{2} - 2 od(w)^2 + 2 id(w)od(w))]	\\
&=&    (n-2)\sum_{w \in V}[ 2(n-1) \left(od(w) - \frac{n-1}{2} \right)^2 + \frac{n-1}{2} - 2 od(w)^2 + 2 id(w)od(w))]	\\
&=&    (n-2)\sum_{w \in V}[ 2(n-1) \left(od(w) - \frac{n-1}{2} \right)^2 + \frac{n-1}{2} - 2 od(w)(od(w)-id(w))]	\\
&=&    (n-2)\sum_{w \in V}[ 2(n-1) \left(od(w) - \frac{n-1}{2} \right)^2 + \frac{n-1}{2} - 2 od(w)(2od(w) - (n-1))]	\\
&=&    (n-2)\sum_{w \in V}[ 2(n-1) \left(od(w) - \frac{n-1}{2} \right)^2 + \frac{n-1}{2} - 4 od(w)^2 + 2(n-1)od(w)]	\\
&=&    (n-2)\sum_{w \in V}[ 2(n-3) \left(od(w) - \frac{n-1}{2} \right)^2 + \frac{n-1}{2} + (n-1)^2 - 2(n-1)od(w)] \\
&=&    2(n-2)(n-3)\sum_{w \in V} \left(od(w) - \frac{n-1}{2} \right)^2 + (n-2)\left(n\left(\frac{n-1}{2}\right) + n (n-1)^2 - 2(n-1){n \choose 2}\right) \\
&=&    2(n-2)(n-3)\sum_{w \in V} \left(od(w) - \frac{n-1}{2} \right)^2  + 3 {n \choose 3} 
\end{eqnarray*}
as desired.
\end{proof}

%
%
%
%
%
%
%
%

\section{Generalizations and future directions}
It is unexpected and exciting that $c(n,k)$, the maximum number of directed $k$-cycles in an $n$-vertex tournament, is asymptotically equal to the expected number of these cycles in a random tournament with edge density $p$ when $k{=}3$ and $k{=}5$, but not when $k{=}4$. A natural next direction is to find for which $k$ $c(n,k)$ is asymptotically equal to the expected value, $\displaystyle \frac{(k-1)!}{2^k} {n \choose k}$.

In finding the maximum number of 5-cycles, we made use of the fact that the (exact) number of 5-cycles in any tournament can be written in terms of its edge score sequence (that is, using the values $A(u,v), B(u,v), C(u,v)$, and $D(u,v)$ for each edge $(u,v)$ in the tournament). It is interesting to note that this approach will not work for computing the number of 6-cycles in a tournament, as $c(T,6)$ cannot be written in terms of the edge score sequence. It would be very interesting to find a combinatorial interpretation of the formula for $c(T,5)$ written in terms of the edge score sequence; for instance, 
$$c(T,5) = \frac{3}{4}{n \choose 5} - \frac{1}{8}\sum_{(u,v)\in E} [(C{+}D)(A{-}B)^2 + (A{+}B)(C{-}D)^2] + \frac{1}{4}\sum_{(u,v)\in E} (A{+}B)(C{+}D)$$
as discovered above.

\section{Appendix: Verification of Linear Relations}
\label{section:appendix}

Let $T=(V,E)$ be an arbitrary tournament with $V = \{1, 2, \dots , n\}$.  For each $i \in [12]$ let $V_i$ be the 
set of 5-vertex subsets of $V$ that induce a tournament isomorphic to $T_i$. For each 5-vertex subset $S$ of $V$,
let $P_S$ be the set of permutations of $S$ (i.e.\ bijections from $S$ to $S$).

\vspace{.15truein}

Define indicator functions as follows:

$f(i,j)$ is 1 if $(i,j)$ is an edge, and 0 otherwise.

$A(i,j,k)$ is 1 if both $i$ and $j$ have $k$ as an out-neighbor, and 0 otherwise.

$B(i,j,k)$ is 1 if both $i$ and $j$ have $k$ as an in-neighbor, and 0 otherwise.

$C(i,j,k)$ is 1 if $i$ has $k$ as an out-neighbor and $j$ has $k$ as an in-neighbor, and 0 otherwise.

$D(i,j,k)$ is 1 if $i$ has $k$ as an in-neighbor and $j$ has $k$ as an out-neighbor, and 0 otherwise.

\vspace{.15truein}

Observe that for an edge $(i,j)$, we have 
$A(i,j) = \sum_{k=1}^n A(i,j,k)$, $B(i,j) = \sum_{k=1}^n B(i,j,k)$,
$C(i,j) = \sum_{k=1}^n C(i,j,k)$ and $D(i,j) = \sum_{k=1}^n D(i,j,k)$.
Thus, to verify the first equation,

\vspace{.1truein}

$$\sum_{(i,j) \in E} {A(i,j) \choose 2} C(i,j)
= \frac12 \sum_{(i,j) \in E} A(i,j) A(i,j) C(i,j) - \frac12 \sum_{(i,j) \in E} A(i,j) C(i,j)$$

$$= \frac12 \sum_{i=1}^n \sum_{j=1}^n \sum_{k=1}^n \sum_{l=1}^n \sum_{m=1}^n f(i,j) A(i,j,k) A(i,j,l) C(i,j,m) - 
\frac12 \sum_{i=1}^n \sum_{j=1}^n \sum_{l=1}^n \sum_{m=1}^n f(i,j) A(i,j,l) C(i,j,m)$$

\vspace{.1truein}
\noindent
Notice that the terms in the first nested sum are 0, unless $i$, $j$, $k$, $l$ and $m$ are distinct or
($k=l$ and $i$, $j$, $k$ and $m$ are distinct). Thus the above expression is 

$$= \frac12 \sum_{q=1}^{12} \sum_{s=\{i,j,k,l,m\} \in V_q} \sum_{\pi \in P_s} 
f(\pi(i),\pi(j))\cdot A(\pi(i),\pi(j),\pi(k))\cdot A(\pi(i),\pi(j),\pi(l))\cdot C(\pi(i),\pi(j),\pi(m))$$
$$+\frac12 \sum_{i=1}^n \sum_{j=1}^n \sum_{k=l=1}^n \sum_{m=1}^n f(i,j) A(i,j,k) A(i,j,l) C(i,j,m)
- \frac12 \sum_{i=1}^n \sum_{j=1}^n \sum_{l=1}^n \sum_{m=1}^n f(i,j) A(i,j,l) C(i,j,m)$$

\vspace{.1truein}
\noindent
The nested sums on the second line of the preceding expression cancel, since $A(i,j,l) A(i,j,l)= A(i,j,l)$
for all $i$, $j$ and $l$. Hence the preceding expression is 

$$= \frac12 \sum_{q=1}^{12} \sum_{s=\{i,j,k,l,m\} \in V_q} \sum_{\pi \in P_s} 
f(\pi(i),\pi(j))\cdot A(\pi(i),\pi(j),\pi(k))\cdot A(\pi(i),\pi(j),\pi(l))\cdot C(\pi(i),\pi(j),\pi(m))$$

\vspace{.1truein}
\noindent
Since the sum over $\pi \in P_s$ depends only on the isomorphism class of the tournament induced by $s$, this is

$$= \frac12 \sum_{q=1}^{12} A_q(T) \sum_{\pi \in P_{V(T_q)}} 
f(\pi(1),\pi(2))\cdot A(\pi(1),\pi(2),\pi(3))\cdot A(\pi(1),\pi(2),\pi(4))\cdot C(\pi(1),\pi(2),\pi(5))$$

\vspace{.1truein}
\noindent
where the sum over $\pi \in P_{V(T_q)}$ is calculated for the tournament $T_q$ by labeling its vertices 
with $\{1,2,3,4,5\}$ from top to bottom and left to right. Note that every term in the sum is 0 or 1,
so we will list, for each of the 12 tournaments, all of the non-zero terms. The computer code and source
file which automate this procedure can be found at
\verb|www.math.cmu.edu/~jmackey/tourn.f| and \verb| www.math.cmu.edu/~jmackey/tourn5|.

\vspace{.1truein}
For $T_1$, $f(1,3) A(1,3,4) A(1,3,5) C(1,3,2) = 
            f(1,3) A(1,3,5) A(1,3,4) C(1,3,2) = 1$

\vspace{.1truein}
For $T_4$, $f(1,2) A(1,2,3) A(1,2,5) C(1,2,4) = 
            f(1,2) A(1,2,5) A(1,2,3) C(1,2,4) = $

           $f(1,3) A(1,3,4) A(1,3,5) C(1,3,2) =
            f(1,3) A(1,3,5) A(1,3,4) C(1,3,2) = $

           $f(1,4) A(1,4,2) A(1,4,5) C(1,4,3) = 
            f(1,4) A(1,4,5) A(1,4,2) C(1,4,3) = 1$

\vspace{.1truein}
For $T_6$, $f(1,2) A(1,2,4) A(1,2,5) C(1,2,3) = 
            f(1,2) A(1,2,5) A(1,2,4) C(1,2,3) = $

           $f(1,3) A(1,3,2) A(1,3,4) C(1,3,5) = 
            f(1,3) A(1,3,4) A(1,3,2) C(1,3,5) = 1$

\vspace{.15truein}
\noindent
Hence, $\sum_{(i,j) \in E} {A(i,j) \choose 2} C(i,j) = A_1 + 3 A_4 + 2 A_6$, as desired.

\vspace{.2truein}
To verify the second equation, we have $$\sum_{(i,j) \in E} {A(i,j) \choose 2} D(i,j)$$

$$= \frac12 \sum_{q=1}^{12} A_q(T) \sum_{\pi \in P_{V(T_q)}} 
f(\pi(1),\pi(2))\cdot A(\pi(1),\pi(2),\pi(3))\cdot A(\pi(1),\pi(2),\pi(4))\cdot D(\pi(1),\pi(2),\pi(5))$$

\noindent
and the corresponding non-zero terms are

\vspace{.1truein}
For $T_2$, $f(1,3) A(1,3,4) A(1,3,5) D(1,3,2) = 
            f(1,3) A(1,3,5) A(1,3,4) D(1,3,2) = $

           $f(2,1) A(2,1,4) A(2,1,5) D(2,1,3) =
            f(2,1) A(2,1,5) A(2,1,4) D(2,1,3) = $

           $f(3,2) A(3,2,4) A(3,2,5) D(3,2,1) = 
            f(3,2) A(3,2,5) A(3,2,4) D(3,2,1) = 1$

\vspace{.1truein}
For $T_3$, $f(2,1) A(2,1,3) A(2,1,5) D(2,1,4) = 
            f(2,1) A(2,1,5) A(2,1,3) D(2,1,4) = 1$

\vspace{.1truein}
For $T_7$, $f(1,2) A(1,2,3) A(1,2,4) D(1,2,5) = 
            f(1,2) A(1,2,4) A(1,2,3) D(1,2,5) = 1$

\vspace{.1truein}
For $T_8$, $f(1,2) A(1,2,3) A(1,2,4) D(1,2,5) = 
            f(1,2) A(1,2,4) A(1,2,3) D(1,2,5) = 1$

\vspace{.15truein}
\noindent
Hence, $\sum_{(i,j) \in E} {A(i,j) \choose 2} D(i,j) = 3 A_2 + A_3 + A_7 + A_8$, as desired.

\vspace{.2truein}
To verify the third equation, we have $$\sum_{(i,j) \in E} {B(i,j) \choose 2} C(i,j)$$

$$= \frac12 \sum_{q=1}^{12} A_q(T) \sum_{\pi \in P_{V(T_q)}} 
f(\pi(1),\pi(2))\cdot B(\pi(1),\pi(2),\pi(3))\cdot B(\pi(1),\pi(2),\pi(4))\cdot C(\pi(1),\pi(2),\pi(5))$$

\noindent
and the corresponding non-zero terms are

\vspace{.1truein}
For $T_1$, $f(3,5) B(3,5,1) B(3,5,2) C(3,5,4) = 
            f(3,5) B(3,5,2) B(3,5,1) C(3,5,4) = 1$

\vspace{.1truein}
For $T_3$, $f(3,5) B(3,5,1) B(3,5,2) C(3,5,4) = 
            f(3,5) B(3,5,2) B(3,5,1) C(3,5,4) = $
          
           $f(4,5) B(4,5,1) B(4,5,3) C(4,5,2) = 
            f(4,5) B(4,5,3) B(4,5,1) C(4,5,2) = 1$

\vspace{.1truein}
For $T_4$, $f(2,5) B(2,5,1) B(2,5,4) C(2,5,3) = 
            f(2,5) B(2,5,4) B(2,5,1) C(2,5,3) = $

           $f(3,5) B(3,5,1) B(3,5,2) C(3,5,4) =
            f(3,5) B(3,5,2) B(3,5,1) C(3,5,4) = $

           $f(4,5) B(4,5,1) B(4,5,3) C(4,5,2) = 
            f(4,5) B(4,5,3) B(4,5,1) C(4,5,2) = 1$

\vspace{.15truein}
\noindent
Hence, $\sum_{(i,j) \in E} {B(i,j) \choose 2} C(i,j) = A_1 + 2 A_3 + 3 A_4$, as desired.

\vspace{.2truein}
To verify the fourth equation, we have $$\sum_{(i,j) \in E} {B(i,j) \choose 2} D(i,j)$$

$$= \frac12 \sum_{q=1}^{12} A_q(T) \sum_{\pi \in P_{V(T_q)}} 
f(\pi(1),\pi(2))\cdot B(\pi(1),\pi(2),\pi(3))\cdot B(\pi(1),\pi(2),\pi(4))\cdot D(\pi(1),\pi(2),\pi(5))$$

\noindent
and the corresponding non-zero terms are

\vspace{.1truein}
For $T_5$, $f(3,4) B(3,4,1) B(3,4,2) D(3,4,5) = 
            f(3,4) B(3,4,2) B(3,4,1) D(3,4,5) = $

           $f(4,5) B(4,5,1) B(4,5,2) D(4,5,3) =
            f(4,5) B(4,5,2) B(4,5,1) D(4,5,3) = $

           $f(5,3) B(5,3,1) B(5,3,2) D(5,3,4) = 
            f(5,3) B(5,3,2) B(5,3,1) D(5,3,4) = 1$

\vspace{.1truein}
For $T_6$, $f(4,5) B(4,5,1) B(4,5,2) D(4,5,3) = 
            f(4,5) B(4,5,2) B(4,5,1) D(4,5,3) = 1$

\vspace{.1truein}
For $T_7$, $f(4,5) B(4,5,2) B(4,5,3) D(4,5,1) = 
            f(4,5) B(4,5,3) B(4,5,2) D(4,5,1) = 1$

\vspace{.1truein}
For $T_8$, $f(4,3) B(4,3,1) B(4,3,2) D(4,3,5) = 
            f(4,3) B(4,3,2) B(4,3,1) D(4,3,5) = 1$

\vspace{.15truein}
\noindent
Hence, $\sum_{(i,j) \in E} {B(i,j) \choose 2} D(i,j) = 3 A_5 + A_6 + A_7 + A_8$, as desired.

\vspace{.2truein}
To verify the fifth equation, we have $$\sum_{(i,j) \in E} {C(i,j) \choose 2} A(i,j)$$

$$= \frac12 \sum_{q=1}^{12} A_q(T) \sum_{\pi \in P_{V(T_q)}} 
f(\pi(1),\pi(2))\cdot C(\pi(1),\pi(2),\pi(3))\cdot C(\pi(1),\pi(2),\pi(4))\cdot A(\pi(1),\pi(2),\pi(5))$$

\noindent
and the corresponding non-zero terms are

\vspace{.1truein}
For $T_1$, $f(1,4) C(1,4,2) C(1,4,3) A(1,4,5) = 
            f(1,4) C(1,4,3) C(1,4,2) A(1,4,5) = 1$

\vspace{.1truein}
For $T_5$, $f(1,3) C(1,3,2) C(1,3,5) A(1,3,4) = 
            f(1,3) C(1,3,5) C(1,3,2) A(1,3,4) = $

           $f(1,4) C(1,4,2) C(1,4,3) A(1,4,5) =
            f(1,4) C(1,4,3) C(1,4,2) A(1,4,5) = $

           $f(1,5) C(1,5,2) C(1,5,4) A(1,5,3) = 
            f(1,5) C(1,5,4) C(1,5,2) A(1,5,3) = 1$

\vspace{.1truein}
For $T_6$, $f(1,4) C(1,4,2) C(1,4,3) A(1,4,5) = 
            f(1,4) C(1,4,3) C(1,4,2) A(1,4,5) = $

           $f(1,5) C(1,5,2) C(1,5,4) A(1,5,3) = 
            f(1,5) C(1,5,4) C(1,5,2) A(1,5,3) = 1$

\vspace{.15truein}
\noindent
Hence, $\sum_{(i,j) \in E} {C(i,j) \choose 2} A(i,j) = A_1 + 3 A_5 + 2 A_6$, as desired.

\vspace{.2truein}
To verify the sixth equation, we have $$\sum_{(i,j) \in E} {C(i,j) \choose 2} B(i,j)$$

$$= \frac12 \sum_{q=1}^{12} A_q(T) \sum_{\pi \in P_{V(T_q)}} 
f(\pi(1),\pi(2))\cdot C(\pi(1),\pi(2),\pi(3))\cdot C(\pi(1),\pi(2),\pi(4))\cdot B(\pi(1),\pi(2),\pi(5))$$

\noindent
and the corresponding non-zero terms are

\vspace{.1truein}
For $T_1$, $f(2,5) C(2,5,3) C(2,5,4) B(2,5,1) = 
            f(2,5) C(2,5,4) C(2,5,3) B(2,5,1) = 1$

\vspace{.1truein}
For $T_2$, $f(1,5) C(1,5,3) C(1,5,4) B(1,5,2) = 
            f(1,5) C(1,5,4) C(1,5,3) B(1,5,2) = $

           $f(2,5) C(2,5,1) C(2,5,4) B(2,5,3) =
            f(2,5) C(2,5,4) C(2,5,1) B(2,5,3) = $

           $f(3,5) C(3,5,2) C(3,5,4) B(3,5,1) = 
            f(3,5) C(3,5,4) C(3,5,2) B(3,5,1) = 1$

\vspace{.1truein}
For $T_3$, $f(1,5) C(1,5,3) C(1,5,4) B(1,5,2) = 
            f(1,5) C(1,5,4) C(1,5,3) B(1,5,2) = $

           $f(2,5) C(2,5,1) C(2,5,3) B(2,5,4) = 
            f(2,5) C(2,5,3) C(2,5,1) B(2,5,4) = 1$

\vspace{.15truein}
\noindent
Hence, $\sum_{(i,j) \in E} {C(i,j) \choose 2} B(i,j) = A_1 + 3 A_2 + 2 A_3$, as desired.

\vspace{.2truein}
To verify the seventh equation, we have $$\sum_{(i,j) \in E} {D(i,j) \choose 2} A(i,j)$$

$$= \frac12 \sum_{q=1}^{12} A_q(T) \sum_{\pi \in P_{V(T_q)}} 
f(\pi(1),\pi(2))\cdot D(\pi(1),\pi(2),\pi(3))\cdot D(\pi(1),\pi(2),\pi(4))\cdot A(\pi(1),\pi(2),\pi(5))$$

\noindent
and the corresponding non-zero terms are

\vspace{.1truein}
For $T_3$, $f(4,2) D(4,2,1) D(4,2,3) A(4,2,5) = 
            f(4,2) D(4,2,3) D(4,2,1) A(4,2,5) = 1$

\vspace{.1truein}
For $T_8$, $f(5,1) D(5,1,2) D(5,1,3) A(5,1,4) = 
            f(5,1) D(5,1,3) D(5,1,2) A(5,1,4) = 1$

\vspace{.1truein}
For $T_9$, $f(5,1) D(5,1,3) D(5,1,4) A(5,1,2) = 
            f(5,1) D(5,1,4) D(5,1,3) A(5,1,2) = 1$

\vspace{.1truein}
For $T_10$, $f(1,3) D(1,3,2) D(1,3,5) A(1,3,4) = 
             f(1,3) D(1,3,5) D(1,3,2) A(1,3,4) = 1$

\vspace{.15truein}
\noindent
Hence, $\sum_{(i,j) \in E} {D(i,j) \choose 2} A(i,j) = A_3 + A_8 + A_9 + A_{10}$, as desired.

\vspace{.2truein}
To verify the eighth equation, we have $$\sum_{(i,j) \in E} {D(i,j) \choose 2} B(i,j)$$

$$= \frac12 \sum_{q=1}^{12} A_q(T) \sum_{\pi \in P_{V(T_q)}} 
f(\pi(1),\pi(2))\cdot D(\pi(1),\pi(2),\pi(3))\cdot D(\pi(1),\pi(2),\pi(4))\cdot B(\pi(1),\pi(2),\pi(5))$$

\noindent
and the corresponding non-zero terms are

\vspace{.1truein}
For $T_6$, $f(5,3) D(5,3,2) D(5,3,4) B(5,3,1) = 
            f(5,3) D(5,3,4) D(5,3,2) B(5,3,1) = 1$

\vspace{.1truein}
For $T_8$, $f(3,5) D(3,5,1) D(3,5,4) B(3,5,2) = 
            f(3,5) D(3,5,4) D(3,5,1) B(3,5,2) = 1$

\vspace{.1truein}
For $T_9$, $f(4,5) D(4,5,1) D(4,5,2) B(4,5,3) = 
            f(4,5) D(4,5,2) D(4,5,1) B(4,5,3) = 1$

\vspace{.1truein}
For $T_10$, $f(4,5) D(4,5,1) D(4,5,2) B(4,5,3) = 
             f(4,5) D(4,5,2) D(4,5,1) B(4,5,3) = 1$

\vspace{.15truein}
\noindent
Hence, $\sum_{(i,j) \in E} {D(i,j) \choose 2} B(i,j) = A_6 + A_8 + A_9 + A_{10}$, as desired.

\vspace{.2truein}
To verify the ninth equation, we have $$\sum_{(i,j) \in E} A(i,j) B(i,j) C(i,j)$$

$$= \sum_{q=1}^{12} A_q(T) \sum_{\pi \in P_{V(T_q)}} 
f(\pi(1),\pi(2))\cdot A(\pi(1),\pi(2),\pi(3))\cdot B(\pi(1),\pi(2),\pi(4))\cdot C(\pi(1),\pi(2),\pi(5))$$

\noindent
and the corresponding non-zero terms are

\vspace{.1truein}
For $T_1$, $f(2,4) A(2,4,5) B(2,4,1) C(2,4,3) = 1$

\vspace{.1truein}
For $T_2$, $f(1,4) A(1,4,5) B(1,4,2) C(1,4,3) = 
            f(2,4) A(2,4,5) B(2,4,3) C(2,4,1) = $

           $f(3,4) A(3,4,5) B(3,4,1) C(3,4,2) = 1$

\vspace{.1truein}
For $T_5$, $f(2,3) A(2,3,4) B(2,3,1) C(2,3,5) = 
            f(2,4) A(2,4,5) B(2,4,1) C(2,4,3) = $

           $f(2,5) A(2,5,3) B(2,5,1) C(2,5,4) = 1$

\vspace{.1truein}
For $T_7$, $f(2,4) A(2,4,5) B(2,4,1) C(2,4,3) = 1$

\vspace{.1truein}
For $T_8$, $f(1,4) A(1,4,3) B(1,4,5) C(1,4,2) = 
            f(2,3) A(2,3,5) B(2,3,1) C(2,3,4) = $

           $f(2,4) A(2,4,3) B(2,4,1) C(2,4,5) = 1$

\vspace{.1truein}
For $T_{10}$, $f(3,4) A(3,4,5) B(3,4,1) C(3,4,2) = 1$

\vspace{.15truein}
\noindent
Hence, $\sum_{(i,j) \in E} A(i,j) B(i,j) C(i,j) = A_1 + 3 A_2 + 3 A_5 + A_7 + 3 A_8 + A_{10}$, as desired.

\vspace{.2truein}
To verify the tenth equation, we have $$\sum_{(i,j) \in E} A(i,j) B(i,j) D(i,j)$$

$$= \sum_{q=1}^{12} A_q(T) \sum_{\pi \in P_{V(T_q)}} 
f(\pi(1),\pi(2))\cdot A(\pi(1),\pi(2),\pi(3))\cdot B(\pi(1),\pi(2),\pi(4))\cdot D(\pi(1),\pi(2),\pi(5))$$

\noindent
and the corresponding non-zero terms are

\vspace{.1truein}
For $T_3$, $f(3,4) A(3,4,5) B(3,4,1) D(3,4,2) = 1$

\vspace{.1truein}
For $T_4$, $f(2,3) A(2,3,5) B(2,3,1) D(2,3,4) = 
            f(3,4) A(3,4,5) B(3,4,1) D(3,4,2) = $

           $f(4,2) A(4,2,5) B(4,2,1) D(4,2,3) = 1$

\vspace{.1truein}
For $T_6$, $f(3,2) A(3,2,4) B(3,2,1) D(3,2,5) = 1$

\vspace{.1truein}
For $T_9$, $f(2,3) A(2,3,4) B(2,3,1) D(2,3,5) = 1$

\vspace{.1truein}
For $T_{10}$, $f(2,1) A(2,1,4) B(2,1,5) D(2,1,3) = 
               f(5,2) A(5,2,1) B(5,2,3) D(5,2,4) = 1$

\vspace{.1truein}
For $T_{11}$, $f(2,3) A(2,3,5) B(2,3,1) D(2,3,4) = 
               f(3,4) A(3,4,5) B(3,4,1) D(3,4,2) = $

              $f(4,2) A(4,2,5) B(4,2,1) D(4,2,3) = 1$

\vspace{.1truein}
For $T_{12}$, $f(1,2) A(1,2,4) B(1,2,3) D(1,2,5) = 
               f(2,4) A(2,4,5) B(2,4,1) D(2,4,3) = $

              $f(3,1) A(3,1,2) B(3,1,5) D(3,1,4) = 
               f(4,5) A(4,5,3) B(4,5,2) D(4,5,1) = $

              $f(5,3) A(5,3,1) B(5,3,4) D(5,3,2) = 1$

\vspace{.15truein}
\noindent
Hence, $\sum_{(i,j) \in E} A(i,j) B(i,j) D(i,j) = A_3 + 3 A_4 + A_6 + A_9 + 2 A_{10}
 + 3 A_{11} + 5 A_{12}$, as desired.

\vspace{.2truein}
To verify the eleventh equation, we have $$\sum_{(i,j) \in E} A(i,j) C(i,j) D(i,j)$$

$$= \sum_{q=1}^{12} A_q(T) \sum_{\pi \in P_{V(T_q)}} 
f(\pi(1),\pi(2))\cdot A(\pi(1),\pi(2),\pi(3))\cdot C(\pi(1),\pi(2),\pi(4))\cdot D(\pi(1),\pi(2),\pi(5))$$

\noindent
and the corresponding non-zero terms are

\vspace{.1truein}
For $T_{3}$,  $f(1,4) A(1,4,5) C(1,4,3) D(1,4,2) = 
               f(2,3) A(2,3,5) C(2,3,1) D(2,3,4) = 1$

\vspace{.1truein}
For $T_7$, $f(1,3) A(1,3,4) C(1,3,2) D(1,3,5) = 1$

\vspace{.1truein}
For $T_8$, $f(2,5) A(2,5,4) C(2,5,3) D(2,5,1) = 1$

\vspace{.1truein}
For $T_9$, $f(1,3) A(1,3,4) C(1,3,2) D(1,3,5) = 1$

\vspace{.1truein}
For $T_{10}$, $f(3,2) A(3,2,4) C(3,2,5) D(3,2,1) = 
               f(3,5) A(3,5,2) C(3,5,4) D(3,5,1) = 1$

\vspace{.1truein}
For $T_{11}$, $f(1,2) A(1,2,3) C(1,2,4) D(1,2,5) = 
               f(1,3) A(1,3,4) C(1,3,2) D(1,3,5) = $

              $f(1,4) A(1,4,2) C(1,4,3) D(1,4,5) = 1$

\vspace{.15truein}
\noindent
Hence, $\sum_{(i,j) \in E} A(i,j) C(i,j) D(i,j) = 2 A_3 + A_7 + A_8 + A_9 + 2 A_{10}
 + 3 A_{11}$, as desired.

\vspace{.2truein}
To verify the twelfth equation, we have $$\sum_{(i,j) \in E} B(i,j) C(i,j) D(i,j)$$

$$= \sum_{q=1}^{12} A_q(T) \sum_{\pi \in P_{V(T_q)}} 
f(\pi(1),\pi(2))\cdot B(\pi(1),\pi(2),\pi(3))\cdot C(\pi(1),\pi(2),\pi(4))\cdot D(\pi(1),\pi(2),\pi(5))$$

\noindent
and the corresponding non-zero terms are

\vspace{.1truein}
For $T_{6}$,  $f(2,5) B(2,5,1) C(2,5,4) D(2,5,3) = 
               f(3,4) B(3,4,1) C(3,4,2) D(3,4,5) = 1$

\vspace{.1truein}
For $T_7$, $f(3,5) B(3,5,2) C(3,5,4) D(3,5,1) = 1$

\vspace{.1truein}
For $T_8$, $f(5,4) B(5,4,2) C(5,4,1) D(5,4,3) = 1$

\vspace{.1truein}
For $T_9$, $f(2,4) B(2,4,1) C(2,4,3) D(2,4,5) = 1$

\vspace{.1truein}
For $T_{10}$, $f(1,4) B(1,4,2) C(1,4,3) D(1,4,5) = 
               f(2,4) B(2,4,3) C(2,4,1) D(2,4,5) = 1$

\vspace{.1truein}
For $T_{11}$, $f(2,5) B(2,5,4) C(2,5,3) D(2,5,1) = 
               f(3,5) B(3,5,2) C(3,5,4) D(3,5,1) = $

              $f(4,5) B(4,5,3) C(4,5,2) D(4,5,1) = 1$

\vspace{.15truein}
\noindent
Hence, $\sum_{(i,j) \in E} B(i,j) C(i,j) D(i,j) = 2 A_6 + A_7 + A_8 + A_9 + 2 A_{10}
 + 3 A_{11}$, as desired.

\bibliographystyle{plain}
\bibliography{tournamentsbib}

\begin{thebibliography}{10}

\bibitem{BeinekeHarary}
L.W. Beineke and F.~Harary.
\newblock The maximum number of strongly connected subtournaments.
\newblock {\em Canad.\ Math.\ Bull.}, 8(4), 1965.

\bibitem{ReidBeineke}
L.W. Beineke and K.B. Reid.
\newblock {\em Tournaments}.
\newblock Academic Press, London, 1978.

\bibitem{Berman}
David~M. Berman.
\newblock On the number of {$5$}-cycles in a tournament.
\newblock In {\em Proceedings of the {S}ixth {S}outheastern {C}onference on
  {C}ombinatorics, {G}raph {T}heory, and {C}omputing ({F}lorida {A}tlantic
  {U}niv., {B}oca {R}aton, {F}la., 1975)}, pages 101--108. Congressus
  Numerantium, No. XIV. Utilitas Math., Winnipeg, Man., 1975.

\bibitem{BermanThesis}
D.M. Berman.
\newblock {\em {The Number of 5-cycles in a Tournament}}.
\newblock PhD thesis, University of Pennsylvania, 1973.

\bibitem{BurrRostaConjecture}
S.A. Burr and V.. Rosta.
\newblock On the {R}amsey multiplicities of graph problems and recent results.
\newblock {\em J.\ Graph Theory}, 4:347--361, 1980.

\bibitem{ApproxOfSidorenko}
D.~Conlon, J.~Fox, and B.~Sudakov.
\newblock An approximate version of {S}idorenko's conjecture.
\newblock {\em Geometric and Functional Analysis}, 20(6):1354--1366, 2010.

\bibitem{ErdosConjecture}
P.~Erd\"{o}s.
\newblock On the number of complete subgraphs contained in certain graphs.
\newblock {\em Publ.\ Math.\ Inst.\ Hung.\ Acad.\ Sci.}, 2(A3):459--464, 1962.

\bibitem{Goodman}
A.W. Goodman.
\newblock On sets of acquaintances and strangers at any party.
\newblock {\em Amer.\ Math.\ Monthly}, 66:778--783, 1959.

\bibitem{MultOfSGs}
C.~Jagger, P.~Stovicek, and A.G. Thomason.
\newblock Multiplicities of subgraphs.
\newblock {\em Combinatorica}, 16:123--141, 1996.

\bibitem{KendallSmith}
M.G. Kendall and B.B. Smith.
\newblock On the method of paired comparisons.
\newblock {\em Biometrika}, (31):324--345, 1940.

\bibitem{Lorden}
G.~Lorden.
\newblock Blue-empty chromatic graphs.
\newblock {\em Amer.\ Math.\ Monthly}, 69:114--120, 1962.

\bibitem{MoonThm}
J.~W. Moon.
\newblock On subtournaments of a tournament.
\newblock {\em Canad.\ Math.\ Bull.}, 9(3):297--301, 1966.

\bibitem{MoonApp}
J.~W. Moon.
\newblock {\em Topics on Tournaments}.
\newblock Holt, Rinehart, and Winston, USA, 1968.

\bibitem{Moon3cycles}
J.~W. Moon.
\newblock Uncovered nodes and 3-cycles in tournaments.
\newblock {\em Australasian J. Combin}, pages 157--173, 1993.

\bibitem{KBReid-1989}
K.B. Reid.
\newblock Three problems on tournaments.
\newblock {\em Ann. N. Y. Acad. Sci.}, 576:466–--473, 1989.

\bibitem{Savchenko}
S.~V. Savchenko.
\newblock On 5-cycles and 6-cycles in regular n-tournaments.
\newblock {\em J. Graph Theory}, 83(1):44--77, 2016.

\bibitem{Thomason}
A.G. Thomason.
\newblock Blue-empty chromatic graphs: A disproof of a conjecture of
  {E}rd\"{o}s in {R}amsey {T}heory.
\newblock {\em J.\ London Math.\ Soc.}, 39(2):246--255, 1989.

\end{thebibliography}

\end{document}